\documentclass{amsart}
\usepackage[foot]{amsaddr}
\usepackage{amssymb,latexsym,amsmath,amsfonts,graphicx,xcolor,pb-diagram}
\usepackage[mathscr]{eucal}
\usepackage{tikz, tikz-cd}
\usepackage{relsize}
\usepackage{slashed}

\newcommand{\mpar}{\par \medskip \par }
\newcommand{\Div}{\operatorname{div}}

\newcommand{\Hess}{\operatorname{Hess}}
\newcommand{\Ric}{\operatorname{Ric}}

\numberwithin{equation}{section}

\swapnumbers

\theoremstyle{definition}

\newtheorem{theorem}{Theorem}[section]
\newtheorem{lemma}[theorem]{Lemma}

\newtheorem{proposition}[theorem]{Proposition}

\newtheorem{definition}[theorem]{Definition}

\newtheorem{remark}[theorem]{Remark}
\newtheorem{example}[theorem]{Example}

\makeindex

\begin{document}

\title{Bochner's Identity on Graphs}

\author{Peter March}

\begin{abstract}
We prove an identity on a graph analogous to Bochner's identity on a Riemannian manifold. An auxiliary graph called the complete tangent graph intervenes in the term corresponding to Ricci curvature.
\end{abstract}

\address{Department of Mathematics\\
Rutgers University\\
Hill Center - Busch Campus\\
110 Frelinghuysen Road\\
Piscataway, NJ 08854-8019}

\email{march@math.rutgers.edu}

\maketitle

\section{Introduction}
Bochner's identity states that if $M$ is a Riemannian manifold and $u\colon M\to \mathbb{R}$ is a smooth function then,
$$
\tfrac{1}{2}\Delta|\nabla u|^2=\langle\nabla\Delta u, \nabla u\rangle + |\Hess(u)|^2 + \Ric(\nabla u, \nabla u).
$$
Here $\langle\cdot, \cdot\rangle$ is the Riemannian inner product on the tangent bundle of $M$; $\nabla, \Delta$, and $\Hess$ are the corresponding gradient, Laplace-Beltrami, and Hessian operators; and $\Ric$ is the Ricci curvature of the Levi-Civita connection.

\mpar
We prove an analogous identity on a graph $G$ that states if $\phi\colon G\to \mathbb{R}$ then,
$$
\tfrac{1}{2}\Delta |\nabla\phi |^2 = \langle\nabla\Delta\phi, \nabla\phi\rangle + |\Hess_{\tau G}(\phi)|^2 +B(\nabla\phi, \nabla\phi),
$$
where $B$ is a symmetric, function-valued bilinear form on vector fields on $G$,
\begin{align*}
B(X, Y)(i) = \tfrac{1}{2}\langle X, &(\Delta_{\tau G}-\nabla\Div)Y\rangle_{T_i(G)} + \tfrac{1}{2}\langle Y, (\Delta_{\tau G} -\nabla\Div)X\rangle_{T_i(G)}\\ 
& - 2\langle \nabla_{\tau G}X, \nabla_{\tau G}Y\rangle_{\mathcal{T}_i^2(G)}
\end{align*}
and $i\in V_G$ ranges over the vertices of $G$. 

\mpar
By implication there are analogues of tangent bundles, vector fields, inner products, as well as gradient, divergence, Laplacian, and Hessian operators. The new ingredient is an auxiliary graph $\tau G$ called the complete tangent graph of $G$. 

\mpar
Evidently, $B$ plays the role of Ricci curvature in the identity. We'll see that the real-valued form obtained by integration over $G$ is,
$$
B(X, Y) = \sum_{i\in V_G} B(X,Y)(i)= -\langle X, \left(\Delta_{\tau G} + \nabla\Div\right) Y\rangle_{\mathcal{X}(G)}
$$
and it defines an operator analogous to the dual of the Hodge operator acting on 1-forms of $M$. When the integral of $B$ is specialized to gradient vector fields, it descends to a form on $C(G)$ and we find,
$$
B(\phi, \psi) = \sum_{i\in V_G} B(\nabla\phi,\nabla\psi)(i) = \langle \phi, B\psi\rangle_{C(G)}
$$
where the operator $B$ can be written explicitly in terms of the original graph $G$.

\mpar
There is a large literature on extending ideas from classical differential geometry to discrete settings. To place this article in the literature, recall works like \cite{LY}, \cite{LLY}, \cite{CLY}, and \cite{BCLL} that use so-called curvature-dimension estimates of the Bakry-Emery form,
$$
\Gamma_2(\phi, \psi) = \tfrac{1}{2}[\Delta\Gamma(\phi, \psi) -\Gamma(\phi,\Delta\psi) -\Gamma(\psi, \Delta\phi)]\\
$$ 
to derive eigenvalue estimates of the graph Laplacian. Since,
$$
\Gamma(\phi,\psi) = \tfrac{1}{2}[\Delta(\phi\psi)-\phi\Delta\psi -\psi\Delta\phi] = -\langle\nabla\phi, \nabla\psi\rangle,
$$
(cf. Proposition 3.5.4) it's easy to see that Bochner's identity is equivalent to a formula for the Bakry-Emery form. Lin and Yau calculated $\Gamma_2$ explicitly in Lemma 2.1 of \cite{LY}, so the present work provides a geometric interpretation of their formula. (Note that their normalization of the Laplacian is different from ours so direct comparison requires some care). There are several distinct notions of curvature on graphs and it would be interesting to see how the form $B(X,Y)$ relates to them. For a general review of discrete curvatures see \cite{SSWJ}. 

\mpar
The novelty of our approach is the observation that forming auxiliary graphs from the directed edges of a given graph and iterating them creates unexpected order where none seemed to exist at first sight. This order helps clarify how the neighborhoods of vertices relate to one another and greatly facilitates calculation.

\mpar
In the next two sections we present ideas, notations, and calculations sufficient to make the terms of argument precise.  Many of these items will be familiar, even part of the folklore. But the concept of tangent graphs and results depending on them in an essential way appear to be new. We include some results that are not strictly necessary to prove the identity in order to familiarize the reader with the basic ideas and methods. With these things in hand the proof of the identity is a straightforward calculation. The arguments are elementary and self-contained.

\section{Tangent Graphs}
Imagine a particle walking from vertex to vertex on a graph by traversing edges in one direction or the other. This suggests we think of directed edges as tangent directions and that some tangent directions precede or suceed others. Let's formalize this idea by constructing an  auxiliary graph, called the tangent graph, whose vertices are directed edges of the given graph and whose edges are certain pairs of directed edges of the given graph - a sort of oriented line graph.

\mpar
Let $G=(V_G, E_G)$ be a finite simple graph with vertex set $V_G$, edge set $E_G$, and adjacency matrix $A$. Thus, $A\colon V_G\times V_G\to \{0,1\}$ is symmetric and
$$
E_G = \{\{i,j\}\mid A(i,j)=1\}.
$$

If $\{i,j\}$ is an edge we use the symbol $ij$ to denote the edge directed from $i$ to $j$ and $ji$ to denote the edge with the opposite direction. So, $ij$ is shorthand for the ordered pair $(i,j)\in V_G\times V_G$ when $A(i,j)=1.$

\mpar
Let $V_* = \{ ij\mid \{i,j\}\in E_G\}$ be the vertex set of the tangent graph and consider what a natural adjacency relation between the vertices might be. Clearly, $\{ij, jk\}$ should be an edge of $tG$ because it represents forward motion from $i$ to $j$ and then another forward motion from $j$ to $k$ so we should start with the edge set $E_ *= \{ \{ij, jk\}\mid ij, jk \in V_*\}$. Our approach is to begin with the graph $G_* = (V_*, E_*)$ and consider adding edges to $E_*$ so that certain natural functions on $V_G$ and $V_*$ become graph homomorphisms. We shall see this leads to a pair of tangent graphs $tG\subset\tau G$ where $tG$ and $\tau G$ are extremal in various senses. We define these graphs formally, then justify the definitions.

\begin{definition} 
1. Let $G=(V_G, E_G)$ be a finite simple graph and let $G_*=(V_*, E_*)$ be the graph having vertex set,
$$ 
V_* = \{ij\mid \{i, j\}\in E_G\}
$$
and edge set,
$$
E_* = \{ \{ij, jk\}\mid ij, jk\in V_*\}.
$$

\mpar
2. Let $\sigma\colon V_*\to V_*$ be the \textit{involution} $\sigma(ij)=ji$ and let $\pi, \pi_+\colon V_*\to V_G$ be the \textit{projections},
$$
\pi(ij)=i \quad \text{and}\quad \pi_+(ij)=j,
$$

so that $\pi_+=\pi\circ\sigma$.

\mpar
3. If $u\in V_{tG}$ is a directed edge then we say $\pi(u)$ is the \textit{base point} and $\pi_+(u)$ is the \textit{end point} of $u$. 

\mpar
4. We say that a function $s\colon V_G\to V_*$ is a \textit{vertex section} provided $\pi\circ s (i) = i$ for all $i\in V_G$. The set of all vertex sections is denoted $S_G$.

\mpar
5. The \textit{tangent graph} of $G$ is the graph $tG$ having vertex set $V_{tG}= V_*$ and edge set
$$
E_{tG} = \{\{u,v\}\mid u, v\in V_*,\,\,\pi_+(u)=\pi(v) \,\, \text{or}\,\, \pi(u)=\pi_+(v)\}.
$$

\mpar
6. The \textit{complete tangent graph} of $G$ is the graph $\tau G$ having vertex set $V_{\tau G}=V_*$ and edge set
$$
E_{\tau G} = \{\{u, v\}\mid u, v\in V_*,\,\,\{\pi(u), \pi(v)\}\in E_G\}.
$$
\end{definition}

\begin{remark}
In other words, two directed edges of $G$ are adjacent in $tG$ if and only if the end point of one directed edge is the base point of the other. On the other hand, two directed edges are adjacent in $\tau G$ if and only if their base points are adjacent in $G$.
\end{remark}

\begin{proposition}
1. Let $\mathcal{A}$ be the set of graphs $H=(V_*, E_H)$ such that $E_*\subset E_H$ and $\sigma\colon H\to H$ is a homomorphism. Then,
$$
\bigcap_{H\in\mathcal{A}} E_H = E_{tG},
$$
and the involution $\sigma\colon tG\to tG$ as well as the projections $\pi, \pi_+\colon tG\to G$ are graph homomorphisms.

\mpar
2. Let $\mathcal{B}$ be the set of graphs $H=(V_*, E_H)$ such that $\pi\colon tG\to G$ is a homomorphism. Then,
$$
\bigcup_{H\in\mathcal{B}} E_H =  E_{\tau G}.
$$

3. Let $\mathcal{C}$ be the set of graphs $H=(V_*, E_H)$ such that every vertex section $s\in S_G$ is a homomorphism $s\colon G\to H.$ Then,
$$
\bigcap_{H\in\mathcal{C}} E_H =  E_{\tau G}.
$$

4. $tG$ is a subgraph of $\tau G$ and $\pi\colon \tau G\to G$ is a homomorphism but neither $\pi_+\colon \tau G\to G$ nor $\sigma\colon \tau G\to \tau G$ are homomorphisms, in general.

\mpar
5. $tG = \tau G$ if and only if $G$ is a star graph $K_{1,n},\, n\geq 1.$
\end{proposition}

\begin{proof}
If $E_*\subset E_H$ and $\sigma\colon H\to H$ is a homomorphism then $E_H$ must contain
\begin{align*}
E_*^{\prime} &= \{ \{\sigma(ij), \sigma(jk)\} \mid ij, jk\in V_*\}\\
&= \{ \{ji, kj\} \mid ij, jk\in V_*\}.
\end{align*}

Observe that $E_* = E_*^{\prime} = E_{tG}$ which implies $tG\in \mathcal{A}.$ It also implies if $H\in \mathcal{A}$ then $E_{tG}\subset E_H$ hence $\cap_{H\in\mathcal{A}} \,E_H = E_{tG}$.

\mpar
To see that $\pi$ is a homomorphism suppose $\{ij,kl\}\in E_{tG}$ and $j=k$. Then,
$$
\{\pi(ij), \pi(kl)\} = \{i, k\} = \{ i,j\}\in E_G.
$$
If, instead, $i=l$ then we have,
$$
\{\pi(ij), \pi(kl)\} = \{i, k\} = \{ l,k\}\in E_G.
$$
Thus, $\pi$ and hence $\pi_+ =\pi\circ\sigma$ are homomorphisms.

\mpar
Let $H\in\mathcal{B}$ so that $\pi\colon H\to G$ is a homomorphism. Then $\{ij, kl\}\in E_H$ implies $\{\pi(ij),\pi(kl)\}=\{i,k\}\in V_G.$ But then, by definition, $\{ij,kl\}\in E_{\tau G}$ and therefore $E_H\subset E_{\tau G}.$ On the other hand, $\tau G\in\mathcal{B}$ so $\bigcup_{H\in \mathcal{B}}E_H=E_{\tau G}.$

\mpar
Let $H\in\mathcal{C}$ and $s\in S_G$ be a vertex section. Since $s$ is a homomorphism we have $\{s(i), s(j)\}\in E_H$ whenever $\{i,j\}\in E_G.$ Now, $\pi(s(i))=i$ and $\pi(s(j))=j$ so $s(i)=il$ and $s(j)=jk$ for some vertices $l$ and $k$ such that $\{i, l\}, \{j, k\}\in E_G.$  Observe,
$$
\{ \{ il, jk\}\mid \{i,j\}\in E_G\} = E_{\tau G}
$$
which implies  $\tau G\in\mathcal{C}.$ It also implies if $H\in\mathcal{C}$ then $E_{\tau G}\subset E_H$ hence $\cap_{H\in\mathcal{C}}\, E_H = E_{\tau G}.$

\mpar
It follows from the definitions that $E_{tG}\subset E_{\tau G}$ and $\pi\colon\tau G\to G$ is a homomorphism. Simple examples like a path of length three show that neither $\sigma$ nor $\pi_+$ are homomorphisms, in general.

\mpar
If $G$ contains a triangle or any cycle of length three or more, then it contains incident edges $\{i,j\}, \{ j,k\},$ and $\{k,l\}$ such that $l\neq j$ and $i\neq k$. This implies $\{ji, kl\}\in E_{\tau G}$ but $\{ji, kl\}\notin E_{tG}$. Thus, if $tG=\tau G$ then $G$ must be a tree. But if $G$ has path of length three or more then the agument above shows $E_{tG}\neq E_{\tau G}.$ Therefore $G$ has paths of length at most two, meaning $G$ is a star graph, and in this case it's easy to see that $tG=\tau G.$
\end{proof}

\begin{remark}
1. In geometric terms, the adjacency relation of $tG$ admits of forward translations, backward translations, and reflections of directed edges of $G$. Specifically,  if $ij$ is in an edge of $tG$ then the edge represents either a forward translation $\{ij, jk\},$ a backward translation $\{ij, ki\},$ or a reflection $\{ij, ji\}.$ It is the smallest such graph on the directed edges of $G$ and admits of no other motions. 

\mpar
2. In contrast, the adjacency relation of $\tau G$ only requires that the base point of a directed edge move along an edge of $G$ and does not require the end point of a directed edge do so too, as is the case with $tG$. These extra degrees of freedom allow $\tau G$ to contain multiple embedded copies of $G$, one for each vertex section, and it is the smallest graph on the directed edges of $G$ with this property.

\mpar
3. It's helpful to reiterate the general notation for vertices and edges in tangent graphs. Recall that $V_{tG}=V_{\tau G}$. We have $\{u, v\}\in E_{tG}$ if and only if $\pi_+(u)=\pi(v)$ or $\pi(u)=\pi_+(v)$ and $\{u,v\}\in E_{\tau G}$ if and only if $\{\pi(u), \pi(v)\}\in E_G.$ Note that we often write $V_{tG}$ where, strictly speaking, we should write $V_{\tau G}$ because we want to emphasize the vertex sets are equal.

\mpar
4. We often make arguments involving two or more tangent graphs and functions defined on them. If confusion is likely to arise we attach appropriate subscripts or superscripts to $\sigma, \pi, \pi_+$ and $\deg$ for clarity's sake. But frequently, to reduce notational clutter, we refrain from labelling them when context makes usage clear. 
\end{remark}

It's helpful to recall a simple fact about how graph homomorphisms depend on the edge sets of the domain and range graphs.

\begin{lemma}
Let  $H, H^{\prime}$ and $K, K^{\prime}$ be graphs such that
$$
V_H=V_{H^{\prime}}, \quad V_K=V_{K^{\prime}}, \quad E_{H^{\prime}}\subset E_H, \quad\text{and}\quad E_K\subset E_{K^{\prime}}.
$$
If $\phi\colon H \to K$ is a graph homomorphism then so is $\phi\colon H^{\prime}\to K^{\prime}.$
\end{lemma}

\begin{proof}
The condition for $\phi$ to be a homomorphism is the implication: if $\{i,j\}\in E_H$ then $\{\phi(i), \phi(j)\}\in E_K$. Thus, adding edges to $K$ or deleting edges from $H$ leaves the truth of the implication unchanged.
\end{proof}

The next result shows the assignments $G\to tG$ and $G\to \tau G$ are functorial.

\begin{proposition}
Let $h\colon H\to K$ be a graph homorphism and let the \textit{differential} $dh\colon V_{tH}\to V_{tK}$ be defined by the rule $dh(u)=h(\pi(u))h(\pi_+(u))$. Then,
$$
dh\colon tH\to tK,\quad dh\colon\tau H\to \tau K,\quad \text{and}\quad dh\colon tH\to \tau K
$$
are also graph homomorphisms.
\end{proposition}
 
\begin{proof}
It's simple but worthwhile to check that $dh$ does map $V_{tH}$ to $V_{tK}.$ Since $h$ is a homomorphism, $\{i,j\}\in E_H$ implies $\{h(i), h(j)\}\in E_K,$ hence $ij\in V_{tH}$ implies $dh(ij) = h(i)h(j)\in V_{tK}$.

\mpar
First, we want to show that $\{u, v\}\in E_{tH}$ implies $\{dh(u), dh(v)\}\in E_{tK}.$  Assuming $\{u,v\}\in E_{tH},$ then either $\pi_+(u)=\pi(v)$ or $\pi(u)=\pi_+(v)$. In the former case, 
\begin{align*}
\pi_+(dh(u)) & =\pi_+(h(\pi(u))h(\pi_+(u)))=h(\pi_+(u))\,\,\text{and,}\\
\pi(dh(v)) & =\pi(h(\pi(v))h(\pi_+(v)))=h(\pi(v))=h(\pi_+(u)).
\end{align*}
It follows that $\pi_+(dh(u))=\pi(dh(v))$ and therefore $\{dh(u), dh(v)\}\in E_{tK}.$ In the latter case, 
\begin{align*}
\pi(dh(u)) & =\pi(h(\pi(u))h(\pi_+(u)))=h(\pi(u))\,\,\text{and,}\\
\pi_+(dh(v)) &=\pi_+(h(\pi(v))h(\pi_+(v)))=h(\pi_+(v))=h(\pi(u)).
\end{align*}
It follows that $\pi(d\phi(u))=\pi_+(d\phi(v))$ and therefore $\{dh(u), dh(v)\}\in E_{tK}.$ Thus, $dh\colon tH\to tK$ is a homomorphism.

\mpar
Next, we want to show that $\{u,v\}\in E_{\tau H}$ implies $\{dh(u), dh(v)\}\in E_{\tau K}.$ Assuming $\{u,v\}\in E_{\tau G}$ we have,
\begin{align*}
\pi(dh(u)) & =\pi(h(\pi(u))h(\pi_+(u))) = h(\pi(u))\,\,\text{and,}\\
\pi(dh(v)) & =\pi(h(\pi(v))h(\pi_+(v))) = h(\pi(v)).
\end{align*}
Since $h$ is a homomorphism, $\{\pi(dh(u)),\pi(dh(v))\} = \{h(\pi(u)), h(\pi(v))\}\in E_K$ and therefore $\{dh(u), dh(v)\}\in E_{\tau G}$. Thus, $dh\colon\tau H\to\tau K$ is a homomorphism.

\mpar
An appeal to Lemma 2.4 shows that $dh\colon tG\to\tau G$ is also a homomorphism.
\end{proof}

\begin{remark}
Let $t^2G=t(tG)$ denote the tangent graph of the tangent graph of $G,$ We frequently use the fact that $d\pi_+, d\pi\colon t^2G\to tG$ and $d\sigma\colon t^2G\to t^2G$ are graph homorphisms.
\end{remark}

It's helpful to work out some examples in detail. We'll see that $tK_{1,3}= K_{3,3}$ and the tangent graph of a triangle is a triangular wedge.

\begin{example}
Let $G=K_{1,3}$ be the complete bipartite graph having vertices $V_G=\{1,2,3,4\}$ and edges $E_G=\{\{1, 2\}, \{1, 3\}, \{1,4\}\}.$ (See Figure 1). Clearly, $tG$ has vertex set $V_{tG}= \{12, 21, 13, 31, 14, 41\}.$ According to the definition of adjacency in $tG$, there are no edges between elements of the set $A=\{12, 13, 14\}$ and no edges between elements of the set $B=\{21, 31, 41\},$ so $A$ and $B$ are independent sets in $tG.$ On the otherhand, according to the definition, every element of $A$ is adjacent to every element of $B$, hence $tG= K_{3,3}.$ Note that $\tau G = tG$ by Proposition 2.3.5.
\end{example}

\begin{figure} [h]
\begin{tikzpicture}
\draw[fill=black] (0,0) circle (2pt);
\draw[fill=black] (0,2) circle (2pt);
\draw[fill=black] (1.5,0) circle (2pt);
\draw[fill=black] (1.5,2) circle (2pt);
\draw[fill=black] (3,0) circle (2pt);
\draw[fill=black] (3,2) circle (2pt);

\node at (0, -0.3) {12};
\node at (0, 2.3) {21};
\node at (1.5, 2.3) {31};
\node at (1.5, -0.3) {13};
\node at (3, -0.3) {14};
\node at (3, 2.3) {41};

\draw[thick] (0,0) -- ((0,2) -- (1.5, 0) --(1.5, 2) -- (3,0) -- (3,2) -- (1.5, 0) -- (0,2) -- (3,0) -- (3,2) -- (0,0);
\draw[thick] (0,0) -- (1.5, 2);

\draw[fill=black] (5,0) circle (2pt);
\draw[fill=black] (6,.8) circle (2pt);
\draw[fill=black] (6,2) circle (2pt);
\draw[fill=black] (7,0) circle (2pt);

\node at (5, -0.3) {3};
\node at (6, 0.5) {1};
\node at (6, 2.3) {4};
\node at (7, -0.3) {2};

\draw[thick] (5, 0)-- (6,0.8) -- (6, 2) -- (6, 0.8) -- (7, 0);

\draw[fill=black] (9,0) circle (2pt);
\draw[fill=black] (12,0) circle (2pt);
\draw[fill=black] (10.5, 0) circle (2pt);
\draw[fill=black] (9,2) circle (2pt);
\draw[fill=black] (12,2) circle (2pt);
\draw[fill=black] (10.5,2) circle (2pt);

\node at (9,-0.3) {12};
\node at (12,-0.3) {14};
\node at (10.5,-0.3) {31};
\node at (9, 2.3) {21};
\node at (12,2.3) {41};
\node at (10.5,2.3) {13};

\draw[thick] (9,0) -- (10.5,0) -- (12,0);
\draw[thick] (9,2) -- (10.5,2) -- (12,2);
\draw[thick] (10.5, 0) -- (10.5, 2);
\draw[thick] (12, 0) -- (12, 2);
\draw[thick] (9,0) -- (9, 2);
\draw[thick] (9, 0) -- (12,2);
\draw[thick] (9,2) -- (12, 0);
\end{tikzpicture}
\caption{$K_{1,3}$ flanked by two representations of $K_{3,3}.$ The representation on the right should be compared with the triangular wedge in Figure 2. }
\end{figure}
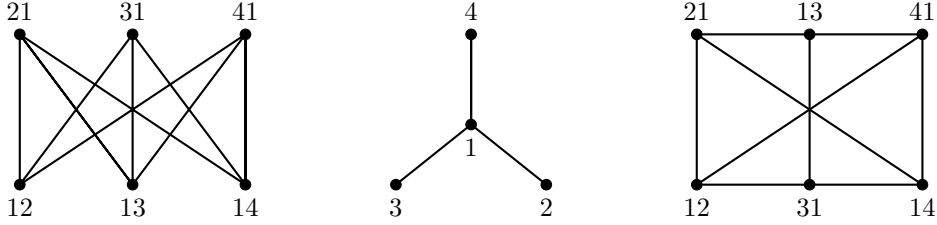

\begin{example}
Let's construct $tG$ and $\tau G$ when $G$ is a triangle with vertices $\{1, 2, 3\}.$ (See Figure 2). Evidently,
$$
V_{tG} = V_{\tau G} = \{ 12, 21, 23, 32, 31, 13\}.
$$

A moment's thought reveals $tG$ contains two triangles $\{ 12, 23, 31\}$ and $\{ 21, 13, 32\},$ corresponding to moving counterclockwise or clockwise around the triangle and three additional edges $\{ 12, 21\}, \{ 23, 32\},$  $\{ 31, 13\}$ corresponding to reflections. Since $tG$ is a subgraph of $\tau G$ all we need to do is add edges between vertices whose base points are adjacent in $G,$ namely $\{12,32\}, \{13, 23\}, \{21, 31\}.$ 

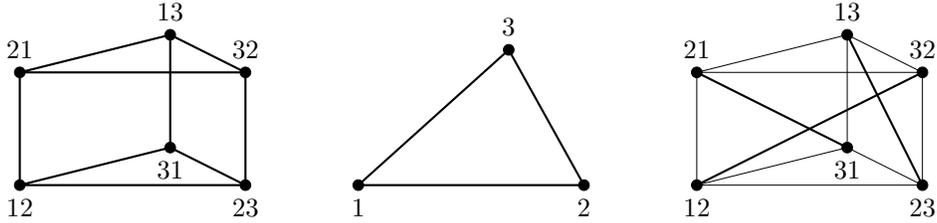
\begin{figure} [h]
\begin{tikzpicture}
\draw[fill=black] (4.5,0) circle (2pt);
\draw[fill=black] (7.5,0) circle (2pt);
\draw[fill=black] (6.5,1.8) circle (2pt);

\node at (4.5,-0.30) {1};
\node at (7.5,-0.3) {2};
\node at (6.5,2.1) {3};

\draw[thick] (4.5,0) -- (7.5,0) -- (6.5,1.8) -- (4.5,0);

\draw[fill=black] (0,0) circle (2pt);
\draw[fill=black] (3,0) circle (2pt);
\draw[fill=black] (2,0.5) circle (2pt);
\draw[fill=black] (0,1.5) circle (2pt);
\draw[fill=black] (3,1.5) circle (2pt);
\draw[fill=black] (2,2) circle (2pt);

\node at (0,-0.3) {12};
\node at (3,-0.3) {23};
\node at (2,0.2) {31};
\node at (0,1.8) {21};
\node at (3,1.8) {32};
\node at (2,2.3) {13};

\draw[thick] (0,0) -- (3,0) -- (2,0.5) -- (0,0) -- (0, 1.5) -- (3, 1.5) --  (2,2) -- (0, 1.5);
\draw[thick] (2, 0.5) -- (2, 2);
\draw[thick] (3, 0) -- (3, 1.5);

\draw[fill=black] (9,0) circle (2pt);
\draw[fill=black] (12,0) circle (2pt);
\draw[fill=black] (11,0.5) circle (2pt);
\draw[fill=black] (9,1.5) circle (2pt);
\draw[fill=black] (12,1.5) circle (2pt);
\draw[fill=black] (11,2) circle (2pt);

\node at (9,-0.3) {12};
\node at (12,-0.3) {23};
\node at (11,0.2) {31};
\node at (9,1.8) {21};
\node at (12,1.8) {32};
\node at (11,2.3) {13};

\draw[very thin] (9,0) -- (12,0) -- (11,0.5) -- (9,0) -- (9, 1.5) -- (12, 1.5) --  (11,2) -- (9, 1.5);
\draw[very thin] (11, 0.5) -- (11, 2);
\draw[very thin] (12, 0) -- (12, 1.5);
\draw[thick] (9,1.5) -- (11, 0.5);
\draw[thick] (11, 2) -- (12,0);
\draw[thick] (9,0) -- (12, 1.5);
\end{tikzpicture}
\caption{From left to right the graphs are $tG, G,$ and $\tau G.$ The edges of $\tau G$ that are in $tG$ are in grey while the edges of $\tau G$ not in $tG$ are in black.}
\end{figure}
\end{example}

Next, we derive formulas for the number of vertices and edges of $tG$ and $\tau G$. They play a role in subsequent calculations and they are interesting in their own right, as they provide checks on the correctness of examples.

\begin{proposition}
Let $G=(V_G, E_G)$ be a finite simple graph with adjacency matrix $A$ and adjacency operator $A\phi(i) = \sum_{j\in V_G} A(i,j)\phi(j)$. Then,
\begin{align*}
|V_{tG}| = |V_{\tau G}| & = 2|E_G|\\
\deg_{tG}(u) & = \deg(\pi_+(u)) + \deg(\pi(u)) -1\\
\deg_{\tau G}(u) & = A\deg(\pi(u))\\
|E_{tG}|+|E_G| & = \sum_{i\in V_G} \deg(i)^2\\
|E_{\tau G}| & = \tfrac{1}{2}\sum_{i\in V_G} \deg(i)\, A\deg(i)
\end{align*}

\end{proposition}

\begin{proof}
The first item is obvious since each edge of $G$ determines two directed edges in $V_{tG}$. Let $u=ij$ be a vertex and let's calculate the degree of $u$ in the respective graphs. Now, $\{u,v\}\in E_{tG}$ if and only if $\pi(v)=j$ or $\pi_+(v)=i$. Hence the set of vertices in $V_{tG}$ adjacent to $u$ is $N_-\cup N_+$ where
$$
N_-=\{ v\in V_{tG}\mid \pi_+(v)=i\},\,\,\text{and}\,\, N_+=\{v\in V_{tG}\mid \pi(v)=j\}.
$$

Clearly, $|N_+| = \deg(j)$,  $|N_-| = \deg(i)$ and $|N_-\cap N_+| =1$ since only the edge $\{ij, ji\}$ lies in both sets, hence $\deg_{tG}(ij) = \deg(i) + \deg(j) -1.$ We have,
\begin{align*}
2|E_{tG}| & = \sum_{u\in V_{tG}}\deg_{tG}(u)\\
& = \sum_{i\in V_G}\sum_{j\in V_G}A(i,j)\deg_{tG}(ij)\\
& = \sum_{i\in V_G}\sum_{j\in V_G}A(i,j)\left[\deg(i)+\deg(j) -1\right]\\
& = 2\sum_{i\in V_G}\deg(i)^2 -\sum_{i\in V_G}\deg(i),
\end{align*}
which is the desired formula, upon rearrangement of terms.

\mpar
Next, recall that $\{u,v\}\in E_{\tau G}$ if only if $\{\pi(u), \pi(v)\}\in E_G$ hence, 
$$
\deg_{\tau G}(u) = \sum_{j\in V_G} A(\pi(u), j)\deg(j) = A\deg(\pi(u)).
$$
It follows that,
$$
2|E_{\tau G}| = \sum_{u\in V_{tG}} \deg_{\tau G}(u) = \sum_{u\in V_{tG}}A\deg(\pi(u)) = \sum_{i\in V_G}\deg(i)\, A\deg(i),
$$
which concludes the proof.
\end{proof}

\begin{remark}
These formulas say that the tangent graph of a triangle has 6 vertices and 9 edges while the complete tangent graph of a triangle has 6 vertices and 12 edges which, in fact, they do.
\end{remark}

We close this section with the observation that $tG$ has the structure of a discrete cubical complex, denoted $t_xG$. This structure clarifies the relationship between the tangent graph $tG$ and the line graph $lG$. We'll see that $tG$ determines $lG$ but not vice versa.

\begin{proposition}
Suppose no connected component of $G$ consists of a single edge or isolated vertex. 

\mpar
1. $tG$ is the union of 4-cycle subgraphs. Specifically, for every pair $u, v\in V_{tG}$ such that $\pi(u)=\pi(v)$ let $C(u, v)$ be the graph with  vertex set $\{ u,\overline{u}, v, \overline{v}\}$ and edge set $\{ \{u, \overline{u}\}, \{ v, \overline{v}\}, \{\overline{u}, v\}, \{u, \overline{v}\} \}.$ Let,

$$
F_{tG} = \{ C(u,v)\mid u, v\in V_{tG},\, u\neq v,\, \pi(u)=\pi(v)\}.
$$

Then,
$$
tG = \bigcup_{C\in F_{tG}} C.
$$

2. Let $t_xG$ be the cubical complex having faces,
$$
f_0(t_xG) = V_{tG},\quad f_1(t_xG) = E_{tG},\quad f_2(t_xG) = F_{tG},
$$
along with $\mathbb{Z}^2$-coefficient chains $C_i =C_i(t_xG),\, 1\leq i\leq 3,$ and differentials,
$$
\partial u =0,\quad \partial \{u, v\} = u+v,\quad \partial C(u, v) = \{u, \overline{u}\} +\{ v, \overline{v}\} + \{\overline{u}, v\}+ \{u, \overline{v}\}.
$$

Then,
$$
0\leftarrow C_0\xleftarrow{\partial} C_1 \xleftarrow{\partial} C_2
$$
is an exact sequence.

\mpar
3. Let $\lambda G$ be the graph with vertex set,
$$
V_{\lambda G} =\{\{u,\overline{u}\}\in E_{tG}\mid u\in V_{tG}\},
$$
and edge set, 
$$
E_{\lambda G} = \{ \{ \{u,\overline{u}\}, \{v,\overline{v}\}\}\mid C(u,v)\in F_{tG}\}.
$$
Then $\lambda G$ is isomorphic to the line graph $lG.$ In particular, $lG$ is a minor of $tG.$
\end{proposition}

\begin{proof}
Let $H=\cup_{C\in F_{tG}} C.$ Then $H$ is a subgraph of $tG$ so we just need to show it's a spanning graph that contains every edge of $tG.$ So, suppose $u=ij\in V_{tG}.$ Since $G$ has no isolated edges there is an edge incident to $\{i,j\},$ say $\{j,k\}$ where $k\neq i.$ Then $u$ is a vertex of $C(ji, jk) \subset H.$ Similarly, if the adjacent edge is $\{i,l\}$ with $l\neq j$ then $u$ is a vertex of $C(ij, il)\subset H,$ and therefore $H$ is a spanning subgraph of $tG.$ Now suppose $\{u,v\}\in E_{tG}$ where $v=kl.$ If $k=j$ then $\{u,v\}$ is an edge of $C(ji,jl)\subset H$ and if $l=i$ then $\{u,v\}$ is an edge of $C(ij, il),$ and therefore $H=tG.$

\mpar
Item 2 is a straightforward verification that $\partial\circ\partial =0$ for chains with $\mathbb{Z}_2$ coefficients, namely,
$$
\partial^2 C(u,v) = 2u + 2\overline{u} + 2v +2\overline{v} = 0.
$$

\mpar
Clearly the correspondence of $\{i,j\}\in E_G$ with $\{ij, ji\}\in E_{tG}$ is a bijection between $V_{lG}$ and $V_{\lambda G}.$ Similarly, edges in $E_{lG}$ are pairs of incident edges in $G$ which are in one-to-one correspondence with pairs of vertices $u, v\in V_{tG}$ such that $\pi(u)=\pi(v)$.  Thus, the condition $C(u,v)\in F_{tG}$ is equivalent to the incidence of edges $\{\pi(u), \pi_+(u)\}$ and $\{\pi(v), \pi_+(v)\}$ which establishes a bijection between $E_{lG}$ and $E_{\lambda G}.$ To see that $lG$ is a minor of $tG$, just observe that the edge $\{\{i,j\}, \{j,k\}\}\in lG$ is isomorphic to the minor of the rectangle $C(ji,jk)$ obtained by contracting the edges $\{ij,ji\}$ and $\{jk,kj\}.$
\end{proof}

\begin{remark}
1. Here's a nice way to visualize $lG$ lying inside $t_xG.$ Imagine each face in $F_{tG}$ is a solid square and $t_xG$ is obtained by gluing these squares along edges of the form $\{u,\overline{u}\}, u\in V_{tG}.$ By item 3, the graph obtained by joining the midpoints of these edges by lines bisecting the squares is isomorphic to the line graph of $G$. 

\mpar
2. $tG$ is a strictly finer invariant than $lG$ in the following sense: $tG$ determines $lG$ but there exist graphs $H$ and $K$ such that $lH\cong lK$ but $tH$ and $tK$ are not isomorphic. In fact $H=C_3$ and $K=K_{1,3}$ are such a pair. It's easy to see that $lC_3 = lK_{1,3}=C_3.$ However, $tC_3$ and $tK_{1,3}$, which were calculated in Examples 2.8 and 2.9, are not isomorphic because the triangular wedge $tC_3$ has genus 0 while $tK_{1,3}=K_{3,3}$ has genus 1. 
\end{remark}

\section{Vector Calculus}
We begin this section by collecting together in one place a number of basic definitions, observations, and notations that are used frequently in the sequel. Generally speaking, these are familiar things expressed in an unfamiliar setting.

\begin{definition}
1. Let $C(G)$ denote the set of \textit{real-valued functions on the graph} $G$. (Strictly speaking these are functions $\phi\colon V_G\to \mathbb{R}$ defined on the vertex set of $G$ so the notation $C(V_G)$ would more accurate, but our notation reflects that of the classical case, which is an important touchstone). Note that in this usage, $C(tG) = C(\tau G)$ since $V_{tG}=V_{\tau G}$. Observe that $C(G)$ has a canonical basis consisting of indicator functions,
$$
e_i(j) = 
\begin{cases}
1, & \text{if $j=i,$}\\
0, & \text{if $i\neq j.$} 
\end{cases}
$$
\mpar
In this notation we have $C(G)=\langle e_i\mid i\in V_G\rangle,$ where the angle brackets mean the vector space defined by the indicated elements.
\mpar
2. The \textit{tangent space} of $G$ at $i\in V_G$ is the vector space,
$$
T_i(G) = \langle e_u \mid u\in V_{tG},\, \pi(u)=i\rangle.
$$
These tangent spaces have a distinguished orthonormal inner product: for all $u, w\in V_{tG}$ such that $\pi(u)=\pi(w) =i,$ we have,
$$
\langle e_u, e_w\rangle = 
\begin{cases}
1, & \text{if $u=v,$}\\
0, & \text{if $u\neq v.$} 
\end{cases}
$$
(While the dual use of angle bracket notation is ambiguous, it is unlikely to cause confusion).

\mpar
3. The \textit{tangent bundle} of $G$ is the coproduct of tangent spaces,
$$
T(G)=\coprod_{i\in V_{tG}} T_i(G).
$$
Elements of the tangent bundle inherit an adjacency relation from $G$, namely $X_i\in T_i(G)$ and $X_j\in T_j(G)$ are adjacent in $T(G)$ if and only if $\{i,j\}\in E_G.$

\mpar
4. A \textit{vector field} on $G$ is a section $X\colon G\to T(G),$ meaning $X_i\in T_i(G)$ for all $i\in V_G$. Every vector field can be written in several ways as a sum, for example,
$$
X=\sum_{i\in V_{tG}} X_i =\sum_{i\in V_G}\sum_{\pi(u)=i} X(u)e_u=\sum_{u\in V_{tG}} X(u)e_u.
$$
The real numbers $X(u)$ are called the \textit{coefficients} of $X$. We have deliberately conflated $X_i,$ thought of as a vector in a tangent space at vertex $i,$ with $X_i$ thought of as a vector field on $G.$ We usually use a subscript only when we want to emphasize $X_i\in T_i(G).$ 

\mpar
5. The \textit{space of vector fields} on $G$ is denoted $\mathcal{X}(G).$ Note that,
$$
C(tG)=\langle e_u\mid u\in V_{tG}\rangle \cong \bigoplus_{i\in V_G}\langle e_u\mid u\in V_{tG}, \pi(u)=i\rangle = \bigoplus_{i\in V_G} T_i(G) = \mathcal{X}(G).
$$
$T_i(G)$ and $\mathcal{X}(G)$ inherit an inner product from $C(tG).$ Typically, we use lower case letters like $f, g,$ and $h$ to denote functions in $C(tG)$ and upper case letters like $X, Y,$ and $Z$ to denote vector fields in $\mathcal{X}(G).$ On the other hand, we always think of the coefficients $X(u)$ of a vector field $X$ as defining a function in $C(tG).$

\mpar
6. For every $u\in V_{tG}$ let $\overline{u}=\sigma(u),$ so that $\overline{ij}=ji.$ For every $X\in\mathcal{X}(G)$ denote by $\overline{X}$ the vector field with coefficients $\overline{X}(u) = X(\overline{u}).$

\mpar
7. Let $\phi\in C(G)$ and $X,Y\in\mathcal{X}(G)$. We define new vector fields $\phi\negthinspace\cdot\negthickspace X$ and $X\negthickspace:\negthinspace Y$ on $G$ by the coefficient formulas,
$$
(\phi\negthinspace\cdot\negthickspace X)(u) = \phi(\pi(u)) X(u), \quad (X\negthickspace :\negthinspace Y)(u) = X(u)Y(u),
$$ 
which turn $\mathcal{X}(G)$ into a commutative algebra over $C(G).$

\end{definition}

\begin{proposition}
1. Let $d\colon C(G)\to C(tG)$ and $\pi,\pi_+\colon C(tG)\to C(G)$ be defined by the rules,
\begin{align*}
d\phi(u) & =\phi(\pi_+(u))-\phi(\pi(u)),\\
\pi f(i) & =\sum_{\pi(u)=i} f(u),\\
\pi_+f(i) & = \sum_{\pi_+(u)=i}f(u) = \sum_{\pi(u)=i} f(\overline{u}).
\end{align*}
Then $d$ and $(\pi_+-\pi)$ are mutually adjoint in the sense that,
$$
\langle d\phi, f\rangle_{C(tG)} = \langle \phi, (\pi_+ - \pi)f\rangle_{C(G)}.
$$

2. Let $\nabla\colon C(G)\to\mathcal{X}(G)$ and $\Div\colon \mathcal{X}(G)\to G$ be defined by the rules,
$$
\nabla\phi = \sum_{u\in V_{tG}}d\phi(u) e_u
$$
and
$$
 \Div X = \sum_{i\in V_G}\left(\pi_+ X(i) - \pi X(i)\right) e_i.
$$
Then, $\nabla$ and $\Div$ are mutually adjoint in the sense that,
$$
\langle\nabla\phi, X\rangle_{\mathcal{X}(G)}=\langle\phi, \Div X\rangle_{C(G)}.
$$

3. Let $\Delta\colon C(G) \to C(G)$ be the Laplace operator $\Delta = \Div\circ\nabla.$ Then,
$$
\Delta\phi(i) = -2\negthickspace\sum_{\pi(u)=i} d\phi(u).
$$
A function $\phi$ is harmonic, that is $\Delta\phi = 0,$ if and only if $\phi$ is constant on each connected component of $G$.
\end{proposition}

\begin{proof}
For item $1$ we have,
\begin{align*}
\langle d\phi, f\rangle_{C(tG)} & = \sum_{u\in V_{tG}}d\phi(u) f(u)\\
& = \sum_{u\in V_{tG}}\negthickspace\phi(\pi_+(u)) f(u) - \sum_{u\in V_{tG}} \negthickspace\phi(\pi(u)) f(u)\\
& = \sum_{i\in V_G}\phi(i)\negthickspace\sum_{\pi_+(v)=i} \negthickspace f(v) - \negthickspace\sum_{i\in V_G} \phi(i)\negthickspace\sum_{\pi(u) =i} \negthickspace f(u)\\
& = \sum_{i\in V_G} \phi(i)\negthickspace\sum_{\pi(u) = i} \negthickspace\left( f(\overline{u}) - f(u)\right)\\
& = \langle \phi, \left(\pi_+ - \pi\right) f\rangle_{C(G)}.
\end{align*}

\mpar
Item $2$ is just a restatement if item $1$ in the language of vector fields, keeping in mind the isomorphism between $C(tG)$ and $\mathcal{X}(G).$
\mpar
To demonstrate item $3$ note that $d\phi(\overline{u}) = -d\phi(u)$ hence,
\begin{align*}
\Delta\phi(i) & = \Div\nabla\phi(i)\\
& = \sum_{u\in V{tG}} \left[d\phi(\overline{u}) - d\phi(u)\right]\\ & = -2\sum_{u\in V_{tG}}d\phi(u).
\end{align*} \
If $\Delta\phi = 0$ then,
$$
0 = \langle\Delta\phi, \phi\rangle = \langle\nabla\phi, \nabla\phi\rangle =\sum_{u\in V_{tG}} |d\phi(u)|^2\negthickspace,
$$
hence $\phi(i) = \phi(j)$ whenever $\{i,j\}\in E_G$. That is, $\phi$ must be constant on each connected compnent of $G$.
\end{proof}

As in the classical case, vector fields can be identified with first order differential operators on functions. The proposition below provides relevant details in the discrete case.

\begin{definition}
Let $G$ be a finite simple graph  and let $d(i,j)$ be the length of the shortest walk between $i$ and $j$. Let $L\colon C(G)\to C(G)$ be a linear transformation so that $L$ has the form,
$$
L\phi(i) = \sum_{j\in V_G}L(i,j)\phi(j).
$$

We say $L$ is a \textit{$k$-th order differential operator} provided, (i) $L(i,j)=0$ whenever $d(i,j)> k$ and (ii) $L\phi = 0$ for every function that is constant on the connected components of $G$. The space of $k$-th order differential operators is denoted $DOp^k(G).$
\end{definition}

\begin{lemma}
Let $\phi, \psi\in C(G)$ and $u\in E_G.$ Then,
$$
d(\phi\psi)(u) = \phi(\pi(u))d\psi(u) + \psi(\pi(u))d\phi(u) + d\phi(u)d\psi(u).
$$
\end{lemma}

\begin{proof}
This identity is a piece of folklore but it's sufficiently important to warrant including a proof. Letting $u=ij$ we have,
\begin{align*}
d\phi(ij) d\psi(ij) & = \left[\phi(j)-\phi(i)\right] \left[\psi(j)-\psi(i)\right]\\
& = \phi(j)\psi(j) - \phi(i)\psi(j) - \psi(i)\phi(j) +\phi(i)\psi(i) \pm 2\phi(i)\psi(i)\\
& = \left[\phi(j)\psi(j) - \phi(i)\psi(i)\right] -\phi(i)\left[\psi(j)-\psi(i)\right] - \psi(i)\left[\phi(j)-\phi(i)\right]\\
& = d(\phi\psi)(ij) -\phi(i) d\psi(ij) - \psi(i) d\phi(ij),
\end{align*}
which is the desired statement.
\end{proof}

\begin{proposition}
Let $X\in\mathcal{X}(G)$ be a vector field on $G$. Then $X$ defines a first order differential operator by the rule,
$$
X\phi(i) = \sum_{\pi(u)=i} X(u)d\phi(u),
$$
or, equivalently,
$$
X\phi = \pi (X\negthickspace:\negthickspace\nabla\phi).
$$
1. The correspondence $X(u) = L(\pi(u), \pi_+(u))$ defines a bijection between $\mathcal{X}(G)$ and $DOp^1(G),$ the space of first order differential operators.

\mpar
2. If $\phi, \psi\in C(G)$ then,
$$
\nabla(\phi\psi) = \phi\negthinspace\cdot\negthickspace\nabla\psi + \psi\negthinspace\cdot\negthickspace\nabla\phi + \nabla\phi\negthickspace:\negthickspace\nabla\psi.
$$

3.  While $X$ is not a derivation, it does satisfy a quadratic extension of the Leibnitz rule,
$$
X(\phi\psi)(i) = \phi(i)X\psi(i) + \psi(i)X\phi(i) + \sum_{\pi(u)=i}X(u) d\phi(u)\psi(u),
$$
or, equivalently, $X\phi\psi = \phi X\psi + \psi X\phi +\pi(X\colon\negthickspace\nabla\phi\colon\negthickspace\nabla\psi).$

\mpar
4. If $X\equiv -2$ then $X=\Delta$ and,
$$
\Delta(\phi^2)(i) = 2\phi(i)\Delta\phi(i) -2 \sum_{\pi(u) = i} |d\phi(u)|^2.
$$
\end{proposition}

\begin{proof}
Let $\psi= 1$ be constant. Then, for all $i\in V_G,$
$$
L\psi(i) = \sum_{j\in V_G} L(i,j) = 0,
$$
and therefore,
\begin{align*}
L\phi(i) & = \sum_{j\in V_G} L(i,j)\phi(j) = \sum_{j\in V_G} L(i,j)\left[\phi(j)-\phi(i)\right]\\
& = \sum_{\pi(u)=i} X(u) d\phi(u) = X\phi(i).
\end{align*}

\mpar
Item 2 follows from Lemma 3.4 and the definitions. The following calculation establishes item 3,
\begin{align*}
X(\phi\psi)(i) & = \sum_{\pi(u)=i} X(u) d(\phi\psi)(u)\\
& = \sum_{\pi(u)=i} X(u)\left[ \phi(i) d\psi(u) + \psi(i) d\phi(u) + d\phi(u) d\psi(u)\right]\\
& = \phi(i) X\psi(i) + \psi(i) X\phi(i) + \sum_{\pi(u) =i}X(u) d\phi(u)\psi(u).
\end{align*}

The statements about the Laplace operator $\Delta$ are self evident. 
\end{proof}

\begin{remark}
1. In this theory, the Laplacian $\Delta$ is a first order operator - a vector field, not a second order operator. While this may seem odd at first, it simply relfects the fact that $\Delta\phi(i)$ depends only on the differences $\phi(j)-\phi(i)$ where $d(i,j)=1.$

\mpar
2. Since the two tangent graphs of $G$ are themselves finite simple graphs they have a pair of tangent graphs of their own, $t(tG)=t^2G$ and $t\tau G,$ as well as gradients $\nabla_{tG}$ and $\nabla_{\tau G},$ divergence operators $\Div_{tG}$ and $\Div_{\tau G}$, and Laplace operators $\Delta_{tG}$ and $\Delta_{\tau G},$ which we think of as acting on $\mathcal{X}(G)$. 
\end{remark}

The next proposition presents an evaluation of two natural self adjoint operators on $C(G)$ defined in terms of these Laplacians. We'll see $\Div\circ\,\Delta_{tG}\circ\nabla$ is a second order operator, since it contains the square of $\Delta,$ whereas $\Div\circ\,\Delta_{\tau G}\circ\nabla$ is a third order operator, since it contains the cube of $\Delta,$ showing once again just how distinct the two notions of tangent graph really are. This result is not strictly necessary for the proof of Bochner's identity, although it is helpful in describing some of its consequences. However, some familiarity with the geometry of $\tau G$ is necessary for the proof and it can be gained here.

\begin{proposition}
Let $D, D_A\in\mathcal{X}(G)$ be vector fields with coefficients,
\begin{align*}
D(u) & = \deg\circ\,\pi_+(u) + \deg\circ\,\pi(u)\\
D_A(u) & =A\,\deg\circ\,\pi_+(u) + A\,\deg\circ\,\pi(u),
\end{align*}
where $A\phi$ is the adjacency operator. Then,
\begin{align*}
\Div\Delta_{tG}\nabla\phi(i) & = \Delta^2\phi(i) -4\Delta\phi(i) -4D\phi(i),\,\text{and,}\\
\Div\Delta_{\tau G}\nabla\phi(i) & = -\tfrac{1}{4}\Delta\left(A\deg\cdot\Delta^2\phi\right)\negthinspace(i) +\tfrac{1}{2}  \Delta\left(\deg\cdot A\deg\cdot\Delta\phi\right)\negthinspace(i) -2 D_A\phi(i).
\end{align*}
\end{proposition}

\begin{remark}
It will be helpful to state and prove the following lemma before taking up the proof of the proposition. We use the notation $\sigma_{tG}$ for the involution on $t^2G$ as well as $\pi_{\tau G}$ for the projection from $t\tau G$ to $\tau G.$ To reduce notational clutter, we will drop these notations when familiarity or context removes any potential ambiguity. Note that item 2 of the lemma will be used in the proof of Bochner's identity.
\end{remark}

\begin{lemma}
1. There is a $\mathbb{Z}_2\times\mathbb{Z}_2$ action on $V_{t^2G}$ generated by the involutions $\sigma_{tG}$ and $d\sigma.$ Specifically, if $\alpha = ij/kl \in V_{t^2G}$ then,
$$
\sigma_{tG}(\alpha) = kl/ij,\qquad d\sigma(\alpha) = ji/lk, \qquad \sigma_{tG}\circ d\sigma(\alpha) = lk/ji = d\sigma\circ\sigma_{tG}(\alpha).
$$

The group action on $V_{t^2G}$ extends to a group action of graph homomorphisms on $t^2G.$

\mpar\mpar
2. For every $i\in V_G$ let,
$$
V^i_{t\tau G} =\{a\in V_{t\tau G}\mid \pi^2(a) =\pi\circ\pi_{\tau G}(a) =i\}
$$

be the set of all vertices of $t\tau G$ based at vertex $i$ of $G$. Then,
\begin{align*}
V^i_{t\tau G} & =\{a\in V_{t\tau G}\mid \pi_{\tau G}(a) = w,\, \pi(w) = i\}\\
& = \{a\in V_{t\tau G}\mid d\pi(a) = u,\, \pi(u) = i\}.
\end{align*}
\end{lemma}

\begin{proof}
These two items really just follow from the definitions, as a few diagrams make clear. 

\begin{figure} [h]
\begin{tikzpicture}
\draw[fill=black] (0,0) circle (2pt);
\draw[fill=black] (1,0) circle (2pt);
\draw[fill=black] (2,0) circle (2pt); 

\node at (0,-0.30) {\textit{i}};
\node at (1,-0.3) {\textit{j}};
\node at (2, -0.3) {\textit{k}};
\node at (1, 0.5) {$\alpha$};

\draw[thick,-Latex] (0,0) -- (0.65,0);
\draw[thick]  (0.6,0) -- (1,0);
\draw[thick,->] (1,0) -- (1.55,0);
\draw[thick] (1.5,0) -- (2,0);

\draw[fill=black] (3,0) circle (2pt);
\draw[fill=black] (4,0) circle (2pt);
\draw[fill=black] (5,0) circle (2pt);

\node at (3,-0.30) {\textit{i}};
\node at (4,-0.3) {\textit{j}};
\node at (5, -0.3) {\textit{k}};
\node at (4,0.5) {$\sigma_{tG}(\alpha)$};

\draw[thick,->] (3,0) -- (3.55,0);
\draw[thick]  (3.5,0) -- (4,0);
\draw[thick,-Latex] (4,0) -- (4.65,0);
\draw[thick] (4.4,0) -- (5,0);

\draw[fill=black] (6,0) circle (2pt);
\draw[fill=black] (7,0) circle (2pt);
\draw[fill=black] (8,0) circle (2pt);

\node at (6,-0.30) {\textit{i}};
\node at (7,-0.3) {\textit{j}};
\node at (8, -0.3) {\textit{k}};
\node at (7,0.5) {$d\sigma(\alpha)$};

\draw[thick] (6,0) -- (6.4,0);
\draw[thick, Latex-]  (6.35,0) -- (7,0);
\draw[thick,-<] (7,0) -- (7.55,0);
\draw[thick] (7.4,0) -- (8,0);

\draw[fill=black] (9,0) circle (2pt);
\draw[fill=black] (10,0) circle (2pt);
\draw[fill=black] (11,0) circle (2pt);

\node at (9,-0.30) {\textit{i}};
\node at (10,-0.3) {\textit{j}};
\node at (11, -0.3) {\textit{k}};
\node at (10,0.5) {$\sigma_{tG}\circ d\sigma(\alpha)$};

\draw[thick,-<] (9,0) -- (9.55,0);
\draw[thick]  (9.35,0) -- (10,0);
\draw[thick,Latex-] (10.35,0) -- (11,0);
\draw[thick] (10,0) -- (10.4,0);
\end{tikzpicture}
\caption{The group acting on $\alpha = ij/jk.$ The directed edge $\pi(\cdot)$ is indicated by a solid arrowhead and the directed edge $\pi_+(\cdot)$ is indicated by an open arrowhead.}
\end{figure}
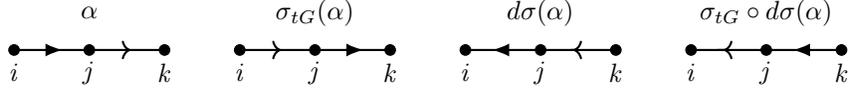

\mpar
That these involutions generate $\mathbb{Z}_2\times\mathbb{Z}_2$ is a straightforward calculation. (See Figure 3). The functions $\sigma_{tG}$ and $d\sigma$ are graph homomorphisms by Propositions 2.1.2 and 2.6. 

\begin{figure} [h]
\quad
\begin{tikzpicture}
\draw[fill=black] (0,1) circle (2pt);
\draw[fill=black] (0.3,0) circle (2pt);
\draw[fill=black] (1.5,0) circle (2pt); 
\draw[fill=black] (2, 0.9) circle (2pt);

\draw[thick] (0,1) -- (0.3,0) -- (1.5,0) -- (2,0.9);

\node at (-0.2,0.5) {\textit{u}};
\node at (2.075,0.5) {\textit{v}};
\node at (0.3,-0.35) {$\pi(u)$};
\node at (1.5,-0.35) {$\pi(v)$};

\draw[thick, ->] (0.3,0) -- (0.15,0.5);
\draw[thick, ->] (1.5,0) -- (1.75, 0.46);


\draw[fill=black] (5,1) circle (2pt);
\draw[fill=black] (5.3,0) circle (2pt);
\draw[fill=black] (6.5,0) circle (2pt); 
\draw[fill=black] (7, 0.9) circle (2pt);

\draw[thick] (5,1) -- (5.3,0) -- (6.5,0) -- (7,0.9);

\node at (4.3,0.5) {$u=\pi(a)$};
\node at (7.75,0.5) {$v=\pi_+(a)$};
\node at (5.95,-0.35) {$w=d\pi(a)$};

\draw[thick, -Latex] (5.3,0) -- (5.134,0.58);
\draw[thick, ->] (6.5,0) -- (6.75, 0.46);
\draw[thick, ->>] (5.3, 0) -- (6.1,0);
\end{tikzpicture}
\caption{The diagram on the left depicts the edge $\{u,v\}$ in $\tau G$ and the diagram on the right depicts the corresponding vertex $a=uv$ in $t\tau G.$ Here $w=d\pi(a) = \pi(u)\pi(v).$ In the right hand diagram, a solid arrowhead indicates $\pi(a),$ an open arrowhead indicates $\pi_+(a),$ and a double arrowheads indicates $d\pi(a).$}
\end{figure}
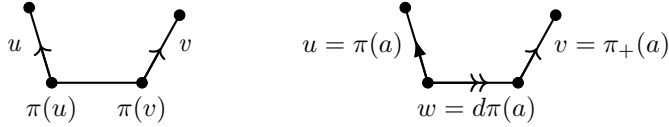

\mpar
The key to item 2 is the observation that if $a\in V_{t\tau G}, w=\pi(a)$ and $u=d\pi(a)$ then $\pi(w)=\pi(u).$ Thus, $V^i_{t\tau G}$ can be indexed in two distinct ways, via the value of $\pi(a)$ or the value of $d\pi(a).$ (See Figure 4).
\end{proof}
Here is the proof of Proposition 3.6.
\begin{proof}
Let's begin by looking closely at the edges of $tG$ or, equivalently, the vertices of $t^2G$ by giving names to their various parts. If $\alpha\in V_{t^2G}$ then,
$$
 u=\pi(\alpha)\in V_{tG},\,\,v=\pi_+(\alpha)\in V_{tG}\,\, \text{and}\,\, \{u,v\}\in E_{tG}.
$$ 
By iteration we have,
\begin{align*}
i &=\pi(u)=\pi\circ\pi(\alpha)\in V_G,\,\, j=\pi_+(u)=\pi_+\circ\pi(\alpha)\in V_G,\,\, \text{and}\,\,\{i,j\}\in E_G,\\
k & =\pi(v)=\pi\circ\pi_+(\alpha)\in V_G,\,\, l=\pi_+(v)=\pi_+\circ\pi_+(\alpha)\in V_G,\,\, \text{and}\,\,\{k,l\}\in E_G.\\
\end{align*} 

We denote $\alpha$ in various ways, as convenient, such as,
$$
\alpha=\pi(\alpha)\pi_+(\alpha) = uv = ij/kl,
$$
where either $j=k$ or $i=l.$
\mpar
Recall that $\alpha$ represents either a forward translation, backward translation, or reflection of $\pi(\alpha)$ and we write their \textit{orientation} $\omega(\alpha) = +1, -1,$ or $0$ accordingly. So, for example, $\omega(\alpha)=1$ means $l=j, k\neq i, \omega(\alpha)=-1$ means $i=k, l\neq j$ and $\omega(\alpha)=0$ means $i=k, j=l,$ as indicated in Figure 5.

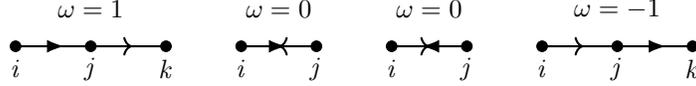
\begin{figure} [h]
\begin{tikzpicture}
\draw[fill=black] (0,0) circle (2pt);
\draw[fill=black] (1,0) circle (2pt);
\draw[fill=black] (2,0) circle (2pt); 

\node at (0,-0.30) {\textit{i}};
\node at (1,-0.3) {\textit{j}};
\node at (2, -0.3) {\textit{k}};
\node at (1, 0.5) {$\omega =1$};

\draw[thick,-Latex] (0,0) -- (0.65,0);
\draw[thick]  (0.6,0) -- (1,0);
\draw[thick,->] (1,0) -- (1.55,0);
\draw[thick] (1.5,0) -- (2,0);

\draw[fill=black] (3,0) circle (2pt);
\draw[fill=black] (4,0) circle (2pt);

\node at (3,-0.30) {\textit{i}};
\node at (4,-0.3) {\textit{j}};
\node at (3.5,0.5) {$\omega=0$};

\draw[thick,<-] (3.5,0) -- (4,0);
\draw[thick,-Latex] (3,0) -- (3.6,0);

\draw[fill=black] (5,0) circle (2pt);
\draw[fill=black] (6,0) circle (2pt);

\node at (5,-0.30) {\textit{i}};
\node at (6,-0.3) {\textit{j}};
\node at (5.5,0.5) {$\omega=0$};

\draw[thick,->] (5,0) -- (5.5,0);
\draw[thick, Latex-] (5.4,0) -- (6,0);

\draw[fill=black] (7,0) circle (2pt);
\draw[fill=black] (8,0) circle (2pt);
\draw[fill=black] (9,0) circle (2pt);

\node at (7,-0.30) {\textit{i}};
\node at (8,-0.3) {\textit{j}};
\node at (9, -0.3) {\textit{k}};
\node at (8,0.5) {$\omega=-1$};

\draw[thick,->] (7,0) -- (7.55,0);
\draw[thick]  (7.5,0) -- (8,0);
\draw[thick,-Latex] (8,0) -- (8.65,0);
\draw[thick] (8.4,0) -- (9,0);
\end{tikzpicture}
\caption{Typical vertices $\alpha\in t^2G$ and their orientations. Solid arrowheads indicate $\pi(\alpha)$ and open arrowheads indicate $\pi_+(\alpha).$}
\end{figure}

Let's compute $\Delta_{tG}\nabla\phi$ first and then take the divergence. Note that we use the following facts without special mention, $(1)\, d\phi(\overline{u}) = -d\phi(u),$ (2) if $ u\neq w$ and $\pi(u)=\pi(w)$ then $u\overline{w}\in V_{t^2G}$ and $\omega(u\overline{w}) =-1,$ (3) if $\omega(\alpha)=1$ then $\pi(\pi_+(\alpha))=\pi_+(\pi(\alpha)),$ and (4) if $\omega(\alpha)=-1$ then $\pi_+(\pi(\alpha))=\pi(\pi(\alpha)).$ We also use the fact that $\Div X(i)=\sum_{\pi(u)=i} X(\overline{u})-X(u).$

\mpar
By the definition of the Laplacian we have,
\begin{align*}
\Delta_{tG}\nabla\phi(u) & = -2\sum_{\substack{\alpha\in V_{t^2G}\\ \pi(\alpha)=u}}\left[\nabla\phi(\pi_+(\alpha))-\nabla\phi(u)\right]\\
& = -2\negthickspace\negthickspace\sum_{\substack{\omega(\alpha) = 0,1\\ \pi(\alpha)=u}}\negthickspace\negthickspace d\phi(\pi_+(\alpha)) -2\negthickspace\negthickspace\sum_{\substack{\omega(\alpha) = -1\\ \pi(\alpha)=u}}\negthickspace\negthickspace d\phi(\pi_+(\alpha)) + 2\negthickspace\negthickspace\sum_{\pi(\alpha)=u}\negthickspace d\phi(u)\\
& = I + II + III.
\end{align*}
Evidently,
$$
III = 2\, \deg_{tG}(u)d\phi(u) = 2(\deg(\pi_+(u)) + \deg(\pi(u)) - 1) d\phi(u)= 2(D(u)-1)d\phi(u).
$$
Next, we have,
$$
 I = -2\negthickspace\negthickspace\sum_{\substack{\omega(\alpha)=0,1\\ \pi(\alpha)=u}}\negthickspace\negthickspace d\phi(\pi_+(\alpha))
 = -2\negthickspace\negthickspace\sum_{\pi(v)=\pi_+(u)}\negthickspace\negthickspace\negthickspace d\phi(v) = \Delta\phi(\pi_+(u)),
 $$
and then we have,
\begin{align*}
II & = -2\negthickspace\negthickspace\sum_{\substack{\omega(\alpha)=-1\\ \pi(\alpha)=u}}\negthickspace\negthickspace d\phi(\pi_+(u)) \pm 2\, d\phi(\overline{u})\\
& = -2\negthickspace\negthickspace\negthickspace\sum_{\pi(w)=\pi(u)}\negthickspace\negthickspace d\phi(\overline{w}) +2\, d\phi(\overline{u})\\
& = \phantom{-} 2\negthickspace\negthickspace\sum_{\pi(w)=\pi(u)}\negthickspace\negthickspace d\phi(w) -2\, d\phi(u)\\ 
& = -\Delta\phi(\pi(u)) -2\, d\phi(u).
\end{align*}

Adding up these terms yields,
$$
\Delta_{tG}\nabla\phi(u) = d\Delta\phi(u) -4\,d\phi(u) + 2D(u) d\phi(u).
$$
Therefore, $\Div\Delta_{tG}\nabla\phi(i)$ equals,
\begin{align*}
\sum_{\pi(u)=i}\negthickspace&\left(\left[d\Delta\phi(\overline{u}) -d\Delta\phi(u)\right]-4\left[d\phi(\overline{u}) -d\phi(u)\right] +2\left[D(\overline{u})\phi(\overline{u})-D(u)d\phi(u)\right]\right)\\
& = -2\sum_{\pi(u)=i}d\Delta\phi(u) +8\sum_{\pi(u)=i}d\phi(u) -4\sum_{\pi(u)=i}D(u)d\phi(u)\\
& = \Delta^2\phi(i) -4\Delta\phi(i) -4D\phi(i).
\end{align*}
Note that in the last line we used the fact that $D$ is even: $D(\overline{u})=D(u).$ 

\mpar
Next, let's look closely at the edges of $\tau G$ or, equivalently, the vertices of $t\tau G$ just as we did with $t^2G$. If $a\in V_{t\tau G}$ then,
$$
u=\pi(a)\in V_{tG},\, v=\pi_+(a)\in V_{tG}, \,\,\text{and}\,\,\{u,v\}\in E_{\tau G}.
$$
This implies $\{\pi(u), \pi(v)\}\in E_G$ hence,
$$
w=d\pi(a) =d\pi(uv) =\pi(u)\pi(v)\in V_{tG}.
$$
These relationships were depicted in Figure 4. As we did before, let's first calculate $\Delta_{\tau G}\nabla\phi$ then take the divergence. We have,
\begin{align*}
\Delta_{\tau G}\nabla\phi(u)  & = -2\sum_{\substack{a\in V_{t\tau G}\\\pi(a)=u}}\left[\nabla\phi(\pi_+(a))-\nabla\phi(u)\right]\\
& = I + II,
\end{align*}
where,
$$
II = 2\sum_{\pi(a)=u} d\phi(u) = 2\deg_{\tau G}(u) d\phi(u) = 2A\deg(\pi(u))d\phi(u).
$$
To evaluate term $I$ it's helpful to decompose the vertex set of $t\tau G$ into disjoint pieces labelled by the value of $\pi(\cdot)$ and $d\pi(\cdot)$ as follows,
$$
\{ a\in V_{t\tau G}\mid \pi(a)=u\} = \bigcup_{\substack{w\in V_{tG}\\\pi(w)=\pi(u)}}\{a\in V_{t\tau G}\mid\pi(a)=u, d\pi(a)=w\}.
$$
Note that that $\pi_+(d\pi(a))=\pi(\pi_+(a))$ which we use in the third line below (See Lemma 3.9.2 and Figure 4). We have,
\begin{align*}
I & = -2\sum_{\substack{a\in V_{t\tau G}\\\pi(a)=u}}\nabla\phi(\pi_+(a))\\
& = -2\sum_{\pi(a)=u}\, \sum_{\pi(w)=\pi(u)}\, \sum_{d\pi(a)=w}\negthickspace d\phi(\pi_+(a))\\
& = -2\sum_{\pi(a)=u}\,\sum_{\pi(w)=\pi(u)}\,\sum_{\pi(v)=\pi_+(w)}\negthickspace\negthickspace\negthickspace  d\phi(v)\\
& = \sum_{\pi(a)=u}\,\sum_{\pi(w)=\pi(u)}\negthickspace \negthickspace \Delta\phi(\pi_+(w))\\
& = \sum_{\pi(a)=u}\,\sum_{\pi(w)=\pi(u)}\negthickspace\negthickspace \left[\Delta\phi(\pi_+(w))\pm\Delta\phi(\pi(u))\right]\\
& = \sum_{\pi(a)=u}\negthickspace \left[-\tfrac{1}{2}\Delta^2\phi(\pi(u)) + \deg(\pi(u))\Delta\phi(\pi(u))\right]\\
& = -\tfrac{1}{2}A\deg(\pi(u))\Delta^2\phi(\pi(u)) +\deg(\pi(u))A\deg(\pi(u))\Delta\phi(\pi(u)).
\end{align*}

Note that we used the fact that $\pi(w)=\pi(u)$ in the second to last line. Finally, we find $\Delta_{\tau G}\nabla\phi(u) = I+II$ is equal to, 
$$
-\tfrac{1}{2}A\deg(\pi(u))\Delta^2\phi(\pi(u))+\deg(\pi(u))A\deg(\pi(u))\Delta\phi(\pi(u)) + 2A\deg(\pi(u))d\phi(u).
$$
\mpar
Let's label these terms $III, IV, V$ and take their divergence one term at a time. We have,
\begin{align*}
\Div III(i) & = -\tfrac{1}{2}\sum_{\pi(u)=i}\left[A\deg(\pi(\overline{u}))\Delta^2\phi(\pi(\overline{u}))-A\deg(i)\Delta^2\phi(i)\right]\\
& = -\tfrac{1}{2}\sum_{\pi(u)=i} [A\deg(\pi_+(u))\Delta^2\phi(\pi_+(u)) - A\deg(i)\Delta^2\phi(i)]\\
& = \tfrac{1}{4}\Delta\left(A\deg\cdot\,\Delta\phi^2\right)(i),\,\,\text{and,}\\
\Div IV(i) & = \negthickspace\sum_{\pi(u)=i}\negthickspace\left[\deg(\pi(\overline{u}))A\deg(\pi(\overline{u}))\Delta^2\phi(\pi(\overline{u}))-\deg(i)A\deg(i)\Delta^2\phi(i)\right]\\
& = \sum_{\pi(u)=i} [\deg(\pi_+(u))A\deg(\pi_+(u))\Delta^2\phi(\pi_+(u)) - \deg(i) A\deg(i)\Delta^2\phi(i)]\\
& = -\tfrac{1}{2}\Delta\left(\deg\cdot\,A\deg\cdot\,\Delta^2\phi\right)(i),\,\,\text{and finally,}\\
\Div V(i) & = 2\sum_{\pi(u)=i}\left[A\deg(\pi(\overline{u}) )d\phi(\overline{u})-A\deg(\pi(u)) d\phi(u)\right]\\
& = -2\sum_{\pi(u)=i}D_A(u)d\phi(u)\\
& = -2D_A\phi(i).
\end{align*}
The proof is finished upon adding up the terms.
\end{proof}

We close this section by defining the Hessian of a function which we think of as  an operator from $C(G)$ to sections of a new bundle on $G$ called the second complete tangent bundle.
\begin{definition}
1. A \textit{vector bundle} on $G$ is the coproduct,
$$
E=\coprod_{i\in V_G} E_i
$$ 
of a set finite dimensional vector spaces, called the \textit{fibers} of $E$, indexed by the vertices of $G$. $E$ inherits an adjacency relation from $G$ by declaring $X_i\in E_i$ is adjacent to $X_j\in E_j$ provided $\{i,j\}\in E_G.$ 

\mpar 
2. A \textit{section} of $E$ is a map $s\colon V_G\to E$ such that $s(i)\in E_i$ for all $ i\in V_G$ and the space of all sections is,
$$
\mathcal{X}(E)\cong\bigoplus_{i\in V_G} E_i
$$ 
(When $E$ is the tangent bundle we use abbreviated notation $\mathcal{X}(T(G))=\mathcal{X}(G))$.  

\mpar
3. The \textit{second tangent bundle} of $G$ is the bundle $T^2(G)$ determined by the fibers,
$$
T^2_i(G) = \langle e_{\alpha}\mid \alpha\in V_{t^2G}, \pi^2(\alpha) =i\rangle
$$
and the \textit{second complete tangent bundle} of $G$ is the bundle $\mathcal{T}^2(G)$ determined by the fibers,
$$
\mathcal{T}^2_i(G) = \langle e_a\mid a\in V_{t\tau G}, \pi^2(a) = i\rangle.
$$
Observe that $T^2(G)$ is a sub-bundle of $\mathcal{T}^2(G)$ in the sense that $T^2_i(G)$ is a vector subspace of $\mathcal{T}^2_i(G)$ for each vertex $i,$ and therefore $\mathcal{X}(T^2(G))$ is a subspace of $\mathcal{X}(\mathcal{T}^2(G)).$

\mpar
4. The $\textit{Hessian}$ is the map $\Hess\colon C(G)\to\mathcal{X}(T^2(G))$ defined by the formula,
$$
\Hess(\phi)_i = \sum_{\pi^2(\alpha)=i}\left[d\phi(\pi_+(\alpha))-d\phi(\pi(\alpha))\right] e_\alpha.
$$

The \textit{complete Hessian} is the map $\Hess_{\tau G}\colon C(G)\to\mathcal{X}(\mathcal{T}(G))$ defined by the formula,
$$
\Hess_{\tau G}(\phi)_i = \sum_{\pi^2(a)=i}\left[d\phi(\pi_+(a))-d\phi(\pi(a))\right]e_a.
$$
When no confusion can arise, we use the abbreviated notation, 
$$
d^2\phi(a) = d\phi(\pi_+(a))- d\phi(\pi(a))=d_{\tau G}\,d\phi(a).
$$

\mpar
5. The fibers $\mathcal{T}^2_i(G)$ carry a natural orthonormal inner product by declaring $\langle e_a, e_b\rangle = 1$ if $a=b$ and zero, otherwise. This inner product is inherited by the subspaces $T^2_i(G).$ Observe that,
$$
|\Hess_{\tau G}(\phi)|^2(i) = \sum_{\pi^2(a)=i} |d^2\phi(a)|^2 = \langle\nabla_{\tau G}\nabla\phi, \nabla_{\tau G}\nabla\phi\rangle_{\mathcal{T}^2_i(G)}.
$$
\end{definition}

\section{Proof of the Identity}
Let $\phi\in C(G)$ and think of, 
$$
|\nabla\phi|^2(i) = \langle \nabla\phi, \nabla\phi\rangle_{T_i(G)} = \sum_{\pi(u)=i}|d\phi(u)|^2
$$ 
as a function on $G$ whose Laplacian we'd like to compute. In doing so, its convenient to think of,
$$
|d\phi(u)|^2 = (d\phi)^2(u)
$$ 
as the square of a function on $C(tG).$ At an important point in the calculation we appeal to the fact that,
\begin{align*}
\{a\in V_{t\tau G}\mid \pi^2(a) =i\} & = \{a\in V_{t\tau G}\mid \pi(a)=u,\, \pi(u)=i\}\\
& = \{a\in V_{t\tau G}\mid d\pi(a)=w,\, \pi(w)=i\}
\end{align*}
to justify a change of variables in a double sum. This was proved in Lemma 3.9.2.
\mpar
We have,
\begin{align*}
\Delta |\nabla\phi|^2(i) & = -2\sum_{\pi(u)=i}d|\nabla\phi|^2(u)\\
& = -2\sum_{\pi(u)=i}\left[\,|\nabla\phi|^2(\pi_+(u)) -|\nabla\phi|^2(i)\,\right]\\
& = -2\sum_{\pi(u)=i}\thickspace\thickspace\thickspace [\negthickspace\negthickspace\negthickspace\sum_{\pi(v)=\pi_+(u)}\negthickspace\negthickspace(d\phi)^2(v) - \sum_{\pi(w)=i}(d\phi)^2(w)\thickspace\thickspace\negthickspace ]\\
& = -2\sum_{\pi(u)=i}\sum_{d\pi(a)=u}[ (d\phi)^2(\pi_+(a) )- (d\phi)^2(\pi(a))]\\
& = -2\sum_{\pi(w)=i}\sum_{\pi(a)=w} [ (d\phi)^2(\pi_+(a))-(d\phi)^2(w) ]\\
& = \sum_{\pi(w)=i}\Delta_{\tau G}(d\phi)^2(w),
\end{align*}
or, equivalently,
$$
\Delta|\nabla\phi|^2 = \pi(\Delta_{\tau G}(\nabla\phi\negthinspace:\negthinspace\negthinspace\nabla\phi)).
$$
\mpar
Recall the product rule for Laplacians in Proposition 3.5.4 as applied to $\Delta_{\tau G}(d\phi)^2(u)$ gives us,
$$
\Delta_{\tau G} (d\phi)^2(u) = 2 d\phi(u)\Delta_{\tau G}d\phi(u) -2\sum_{\pi(a)=u}|d^2\phi(a)|^2.
$$

Then we have,
\begin{align*}
\Delta|\nabla\phi|^2(i) & =\sum_{\pi(u)=i}\Delta_{\tau G}(d\phi)^2(u)\\
& = 2\sum_{\pi(u)=i} d\phi(u)\Delta_{\tau G} d\phi(u) -2\sum_{\pi(u)=i}\sum_{\pi(a) =u}|d^2\phi(a)|^2]\\
& = 2\nabla\phi\cdot\Delta_{\tau G}\nabla\phi(i) -2|\Hess_{\tau G}(\phi)|^2(i),
\end{align*}
where we've used the notation $X_i\cdot Y_i =\langle X_i, Y_i\rangle_{T_i(G)}.$
\mpar
To make this formula look similar to the classical case we add and subtract the term $2\nabla\phi\cdot\nabla\Delta\phi$. Observing that,
$$
\nabla\phi\cdot\Delta_{\tau G}\nabla\phi - \nabla\phi\cdot\nabla\Delta\phi = \nabla\phi\cdot\left(\Delta_{\tau G}-\nabla\Div\right)\nabla\phi
$$
we find,
$$
\tfrac{1}{2}\Delta |\nabla\phi |^2 = \langle\nabla\Delta\phi, \nabla\phi\rangle + |\Hess_{\tau G}(\phi)|^2 + B(\nabla\phi, \nabla\phi)
$$
where,
\begin{align*}
B(\nabla\phi, \nabla\phi)(i) & = \nabla\phi\cdot\left(\left(\Delta_{\tau G}-\nabla\Div\right)\nabla\phi\right)(i) - 2|\Hess_{\tau G}(\phi)|^2(i)\\
& = \langle \nabla\phi, (\Delta_{\tau G}-\nabla\Div)\nabla\phi\rangle_{T_i(G)} - 2\langle\nabla_{\tau G}\nabla\phi,  \nabla_{\tau G}\nabla\phi\rangle_{\mathcal{T}^2_i(G)}.
\end{align*}
Since $\Delta_{\tau G}, \Div,$ and $\nabla_{\tau G}$ are defined on $\mathcal{X}(G)$, this defines a symmetric, function-valued bilinear form on vector fields,
\begin{align*}
B(X, Y)(i) = \tfrac{1}{2}\langle X, &(\Delta_{\tau G}-\nabla\Div)Y\rangle_{T_i(G)} + \tfrac{1}{2}\langle Y, (\Delta_{\tau G} -\nabla\Div)X\rangle_{T_i(G)}\\ 
& - 2\langle \nabla_{\tau G}X, \nabla_{\tau G}Y\rangle_{\mathcal{T}_i^2(G)}
\end{align*}
which plays the role of Ricci curvature in Bochner's identity. 

\begin{remark}
1. Integrating this identity over $G$ we find,
\begin{align*}
\sum_{i\in V_G} B(X, Y)(i) & = \langle X, (\Delta_{\tau G}-\nabla\Div )Y\rangle_{\mathcal{X}(G)} - 2\langle\nabla_{\tau G} X, \nabla_{\tau G}Y\rangle_{\mathcal{X}(\tau G)}\\
& =\langle X, (\Delta_{\tau G}-\nabla\Div) Y\rangle_{\mathcal{X}(G)} - 2\langle X, \Delta_{\tau G} Y\rangle_{\mathcal{X}(G)}\\
& = -\langle X, (\Delta_{\tau G}+\nabla\Div) Y\rangle_{\mathcal{X}(G)}.
\end{align*}

Since,
$$
\Delta_{\tau G} + \nabla\Div =\Div_{\tau G}\nabla_{\tau G}+\nabla_G\Div_G
$$
one sees this is a Hodge operator acting on vector fields on $G$. 

\mpar
2. The bilinear form descends to $C(G)$ as,
\begin{align*}
B(\phi, \psi) & = -\langle \nabla\phi, (\Delta_{\tau G}+\nabla\Div)\nabla\psi\rangle_{\mathcal{X}(G)}\\
& = -\langle\phi,\left(\Div\Delta_{\tau G}\nabla + \Delta^2\right)\psi\rangle_{C(G)}\\
& = \langle\phi, B\psi\rangle_{C(G)},
\end{align*}
where,  according to Proposition 3.7,
$$
B\psi = \tfrac{1}{4}\Delta\left(A\deg\cdot\Delta^2\psi\right)\negthinspace -\tfrac{1}{2}  \Delta\left(\deg\cdot A\deg\cdot\Delta\psi\right)\negthinspace -\Delta^2\psi +2 D_A\psi.
$$

\end{remark}
\subsection*{Acknowledgements} My thanks are due to Aileen Carroll, Bonnie Gordon, Martha March, Larry Susanka, and Gabrielle Wilders for their interest, encouragement and support.

\end{document}